\documentclass{amsart}
\usepackage{a4wide, graphicx, amsmath, amsfonts, amsthm, amssymb, latexsym}   

\usepackage[colorlinks]{hyperref}
\usepackage{tikz}
\usetikzlibrary{arrows,shapes,decorations,decorations.markings}

\hypersetup{linkcolor=blue, urlcolor=blue, citecolor=red}

\let\oldmarginpar\marginpar
\renewcommand\marginpar[1]{\-\oldmarginpar[\raggedleft\footnotesize #1]%
{\raggedright\footnotesize #1}}

\newtheorem{thm}{Theorem}[section]
\newtheorem{lem}[thm]{Lemma}
\newtheorem{prop}[thm]{Proposition}
\newtheorem{cor}[thm]{Corollary}

\newtheorem{exa}[thm]{Example}
\newenvironment{ex}{\begin{exa} \rm}{\end{exa}}

\newcommand{\set}[2]{\{#1:#2\}}
\newcommand{\genset}[1]{\langle#1\rangle}
\newcommand{\rrel}{\mathcal{R}}
\newcommand{\lrel}{\mathcal{L}}
\newcommand{\N}{\mathbb{N}}

\renewcommand{\a}{\boldsymbol\alpha}
\renewcommand{\b}{\boldsymbol\beta}
\renewcommand{\r}{\mathbf{r}}

\renewcommand{\to}{\longrightarrow}

\newcommand{\myloop}[3]{\draw[decoration={
        markings,
        mark=at position 1 with {
            \arrow[scale=2]{>}
        };
    },
    postaction={decorate}] #1 to [out=45, in=-45] #2 to [out=135,in=45] #3 to [out=-135,in=135] #1;}

\tikzset{
vertex/.style={circle,draw,fill,inner sep=0pt,minimum size=1mm},
  ar/.style={
        decoration={markings,mark=at position 1 with {\arrow[scale=2]{>}}},
        ,
        postaction={decorate}
    },
  loop/.style={min distance=15mm,in=45,out=135
     }
}

\begin{document}

\title{Ends of semigroups}
\author{S. Craik, R. Gray, V. Kilibarda, J. D. Mitchell, and N. Ru\v skuc}
\maketitle
\begin{abstract} 
We define the notion of the partial order of ends of the Cayley graph of a semigroup. We prove that the structure of the ends of a semigroup is invariant under change of finite 
generating set and at the same time is inherited by subsemigroups and extensions of finite Rees index. We  prove an analogue of Hopf's Theorem, stating that a group has 1, 2 or 
infinitely many ends, for left cancellative semigroups and that the cardinality of the set of ends is invariant in subsemigroups and extension of finite Green index in left cancellative 
semigroups.
 \end{abstract}


\section{Introduction}

The study of ends in group theory has been extensive and has had widespread influence.  Stallings' Theorem characterising groups with more than one end has been used in such varied topics as distance-transitive graphs \cite{Macpherson}, groups with context-free word problem \cite{Muller+Schupp}, pursuit-evasion problems in infinite graphs \cite{PursuitEvasion} and to describe accessible groups  \cite{Dunwoody}. This paper follows the trend of relating geometric properties of Cayley graphs of semigroups to algebraic properties, see for example \cite{GeoAutoSemi}, \cite{Remmers} and \cite{TropicalJ}.  We consider the notions of ends for a semigroup and try to recover some basic theorems from the theory of ends of groups. 

In this paper we consider a definition for ends of digraphs introduced by Zuther in \cite{Zuther} and apply it to the left and right Cayley graphs of a semigroup. In \cite{Vesna} Jackson and Kilibarda introduce a notion of ends for semigroups which is based on the ends of the underlying undirected graph of the Cayley graph. They prove that the number of ends of a semigroup is invariant under change of finite 
generating set and provide examples of semigroups with $n$ ends in the left Cayley graph and $m$ ends in the right Cayley graph for any prescribed $n,m \in \mathbb{N}$.
We argue that although there are many ways to generalise the notion of ends to a semigroup; by preserving the notion of direction there is a greater chance of interrelating the algebraic structure and the ends.

In the remainder of this section we introduce the relevant definitions and technical results required  to prove the main theorems of this paper. 
In Section \ref{ends of a semigroup}, we prove that the structure of the ends of a semigroup is invariant under change of finite generating set and at 
the same time is 
inherited by subsemigroups and extensions of finite Rees index. 
In Section \ref{left_cancel_sect}, we  prove an analogue of Hopf's Theorem, 
stating that a group has 1, 2 or infinitely many ends,
for left cancellative semigroups and that the cardinality of the set of ends is invariant in subsemigroups and 
extensions of finite Green 
index in left cancellative semigroups.


Let $\Omega$ be any set and let  $\Gamma\subseteq \Omega\times \Omega$. We will refer to $\Gamma$ as a 
\emph{digraph} on 
$\Omega$, the elements of $\Omega$ as the \emph{vertices} of $\Gamma$, and the elements of $\Gamma$ as 
\emph{edges}. A 
\emph{walk} in $\Gamma$ is just a (finite or infinite) sequence $(v_0, v_1, \ldots)$ of (not necessarily distinct) 
vertices such that 
$(v_i, v_{i+1})\in\Gamma$ for all $i$. An \emph{anti-walk} is a sequence  $(v_0, v_1, \ldots)$ of (not necessarily distinct) 
vertices such that 
$(v_{i+1}, v_{i})\in\Gamma$  for all $i$.  
A \emph{path} in $\Gamma$ is just a walk consisting of distinct vertices.   If 
$\a=(\alpha_0, \alpha_1,\ldots, \alpha_n)$ is a walk in $\Gamma$, then the \emph{length} of $\a$ is $n$ and it is straightforward to verify that $\a$ 
contains a path from $\alpha_0$ to $\alpha_n$.
A \emph{ray} in $\Gamma$ is just an infinite path $(v_0, v_1, \ldots)$ such that $(v_i, v_{i+1})\in\Gamma$ for all 
$i$ and an 
\emph{anti-ray} is an infinite path $(v_0, v_1, \ldots)$ such that $(v_{i+1}, v_{i})\in\Gamma$ for all $i$. 
If $\a=(\alpha_0, \alpha_1, \ldots, \alpha_m)$ and $\boldsymbol\beta=(\beta_0, \beta_1, \ldots, \beta_n)$ are 
arbitrary finite 
sequences of elements from $\Omega$, then we denote by $\a^{\smallfrown}\boldsymbol\beta$ the sequence 
$(\alpha_0, \alpha_1, \ldots, \alpha_m, \beta_0, \beta_1, \ldots, \beta_n)$.

The \emph{out-degree} of a vertex $\alpha$ in a digraph $\Gamma$ is just 
$|\set{\beta\in \Omega}{(\alpha,\beta)\in \Gamma}|$. 
A digraph $\Gamma$ is \emph{out-locally finite} if every vertex has finite out-degree. 

If $\Sigma\subseteq \Omega$, then $\Gamma\cap (\Sigma\times \Sigma)$ is the \emph{induced subdigraph} of 
$\Gamma$ on $\Sigma$. If $\Sigma$ and  $\Sigma'$ are infinite subsets of $\Omega$, then we write 
$\Sigma'\preccurlyeq \Sigma$ if there exist infinitely many disjoint paths (including paths of length $0$) in 
$\Gamma$ with initial vertex belonging to $\Sigma'$ and final vertex belonging to $\Sigma$. It is straightforward to verify that $\preccurlyeq$ is reflexive on infinite 
subsets of $\Omega
$ but not necessarily transitive, symmetric, or anti-symmetric.  However, it was shown in \cite{Zuther} that if $\preccurlyeq$ is restricted to the set of rays and anti-rays on $\Gamma$, then it is transitive, 
and hence a 
preorder. 
If $\Sigma'\preccurlyeq \Sigma$ and $\Sigma \preccurlyeq\Sigma'$, then we write $\Sigma  \approx \Sigma'$.   It follows that $\approx$ is an equivalence relation and $
\preccurlyeq$ 
induces a partial order on $\approx$-classes of rays and anti-rays. As such we refer to rays as being \emph{equivalent} if 
they belong to the same $\approx$-class; and \emph{inequivalent} otherwise.
We denote this poset  by $\Omega\Gamma$, and we refer to $\approx$-classes of rays as the \emph{ends} of $\Gamma$.




\begin{lem}\label{walktoray}
Let $\Gamma$ be a digraph on $\Omega$ and let $\a=(\alpha_0, \alpha_1, \ldots)$ be an infinite  walk  (or anti-walk) in  $\Gamma$ such that every vertex of $\a$ occurs 
only finitely many times. Then $\a$ contains a ray  $\r$ (or anti-ray, respectively) such that $\r$ is has infinitely many disjoint paths to every infinite subset of $\{\alpha_0, \alpha_1, \ldots\}$ and every infinite subset of  $\{\alpha_0, \alpha_1, \ldots\}$ has infinitely many disjoint paths to $\r$.
\end{lem}
\proof
We prove that $\a$ contains a ray in the case that $\a$ is a walk;  an analogous argument proves that 
$\a$ contains an anti-ray in the case that it is an anti-walk. 

Let $a(0)=1$ and for every  $i\geq 1$  define $a(i)=\max\set{j\in \N}{\alpha_j=\alpha_{a(i-1)}}+1$, i.e. $\alpha_{a(i)-1}$ is the last appearance of $\alpha_{a(i-1)}$ in $\a$.  
We will 
show that 
$$\r=(\alpha_{a(0)}, \alpha_{a(1)}, \ldots)$$
is the required ray.
Since $(\alpha_{i}, \alpha_{i+1})\in \Gamma$ for all $i$, in particular, 
$(\alpha_{a(i-1)}, \alpha_{a(i)})=(\alpha_{a(i)-1}, \alpha_{a(i)})\in \Gamma.$  
Hence $\mathbf{r}$ is an infinite walk where $\alpha_{a(i)}\not=\alpha_{a(j)}$ for all 
$i,j\in \N$ such that $i\not=j$ and so  $\r$ is a ray.

Let $\Sigma$ be any infinite subset of $\{\alpha_0, \alpha_1, \ldots\}$.  
If infinitely many elements in $\Sigma$ are vertices of $\mathbf{r}$, then $\r$ is equivalent to $\Sigma$.  
If only finitely many elements of $\Sigma$ belong to $\mathbf{r}$, then  
$\Sigma\setminus\{\alpha_{a(0)}, \alpha_{a(1)}, \ldots\}$ is equivalent to $\Sigma$ and so  we may assume 
without loss of generality that 
$\Sigma$ contains no elements in  $\{\alpha_{a(0)}, \alpha_{a(1)}, \ldots\}$.

We define infinitely many disjoint paths from $\Sigma$ to $\r$ by induction. Let $b(0)\in\N$ be any number such that $\alpha_{b(0)}\in \Sigma$.  Then there exists $k(0)\in 
\N$ such 
that
$a(k(0))<b(0)<a(k(0)+1)$ and $\b_0:=(\alpha_{a(k(0))}, \alpha_{a(k(0))+1}, \ldots, a_{b(0)}, \ldots, \alpha_{a(k(0)+1)})$ is a walk from $\alpha_{a(k(0))}$ in $\r$ to $
\alpha_{a(k(0))+1}$ 
in $\r$ via 
$a_{b(0)}\in \Sigma$. Since every finite walk contains a path,  we conclude that there is a path contained in 
$\b_0$ from a vertex of $\r$ to $a_{b(0)}\in \Sigma$ and  a path back from  $a_{b(0)}$  to a vertex of $\r$.

 Suppose that we have defined $b(0), \ldots, b(i-1), k(0), \ldots, k(i-1)\in\N$ and finite walks $\b_0, \b_1, \ldots, \b_{i-1}$ for some $i\geq 1$. 
 Choose $k(i), b(i)\in \N$  so that $b(i)>a(k(i))$, $\alpha_j$ does not equal any vertex in any of $\b_0, \b_1, \ldots, \b_{i-1}$ for all $j>a(k(i))$, and $\alpha_{b(i)}\in \Sigma$. 
Then we define  
$$\b_i=(\alpha_{a(k(i))}, \alpha_{a(k(i))+1}, \ldots, \alpha_{b(i)}, \ldots, \alpha_{a(k(i)+1)}).$$

By construction, if $i\not=j$, then 
$\boldsymbol\beta_i$ and  $\boldsymbol\beta_j$ are disjoint and so we have infinitely many disjoint paths (contained in the $\b_i$) from $\r$ to $\Sigma$ and back, as 
required.
\qed\vspace{\baselineskip}


\begin{lem}\label{concatenation}
Let $\Gamma$ be an out-locally finite digraph on a set $X$ and let $\mathbf{w}_0, \mathbf{w}_1, \ldots$ be finite walks of bounded length in $\Gamma$ with distinct final 
vertices. 
Then every vertex in the sequence $\mathbf{w}_0^\smallfrown\mathbf{w}_1^\smallfrown\cdots$ occurs only finitely many times.
\end{lem}
\proof Let $K\in \N$ be  bound on the lengths of $\mathbf{w}_0, \mathbf{w}_1, \ldots$.
If a vertex $v$ occurs in infinitely many of $\mathbf{w}_0, \mathbf{w}_1, \ldots$, then the set $B$ of vertices that can be reached from $v$ by a path of length at most $K
$ contains the 
final vertex of $\mathbf{w}_i $ for infinitely many $i\in\N$. But the final vertices of the $\mathbf{w}_i$ are distinct and so $B$ is infinite, contradicting the assumption that 
$\Gamma$ is 
out-locally finite. 
\qed


\begin{lem}\label{dominance}
Let $\Gamma$ be an out-locally finite digraph on $\Omega$, let $\Sigma\subseteq \Omega$ be infinite and let $\alpha_0\in\Omega$ such that there is a path from $
\alpha_0$ to every  
$\beta\in \Sigma$. Then there exists a ray  $\bf{r}$ in $\Gamma$ starting at $\alpha_0$ such that $\Sigma\preccurlyeq\bf{r}$.  
\end{lem}
\begin{proof}
We construct $\bf{r}$ recursively. 
Start by setting  $\Sigma_0:=\Sigma$ and let $P_0$ be a set containing precisely one path $q_\beta$ from $\alpha_0$ to $\beta$  for all $\beta\in \Sigma_0$.  Then, 
since $\alpha_0$ 
has finite out-degree and there is a path in $P_0$  from $\alpha_0$ to every $\beta \in\Sigma_0$, there exists  a vertex $\gamma_0$ such that 
$(\alpha_0,\gamma_0)\in \Gamma$ and there is a path $q_{\beta}\in P_0$ from $\alpha_0$ via $\gamma_0$ to every $\beta$ in the infinite subset $\Sigma_1\subseteq 
\Sigma_0$.  

Let $\beta_1\in \Sigma_1$ be fixed and also fix a path
$$p_1=(\delta_1=\alpha_0, \delta_2=\gamma_0, \delta_3, \ldots, \delta_{n-1}, \delta_n=\beta_1).$$
Let $P_1=\set{q_{\beta}\in P_0}{\beta\in \Sigma_1}$.
If $\beta \in \Sigma_1$ is arbitrary and $q_{\beta}\in P_1$, then there exists $i(\beta)\in\mathbb{N}$ such that $\delta_{i(\beta)}$ is the last vertex  belonging to both the 
paths $p_1$ 
and $q_{\beta}$. The number $i(\beta)$ exists since,  in particular, both paths go through $\gamma_0$. By the pigeonhole principle, there exists $m\in\mathbb{N}$ such 
that $2\leq 
m\leq n$ and 
$\Sigma_2=\set{\beta \in \Sigma_1}{i(\beta)=m}$ is infinite. Set $\alpha_1=\delta_m$.  Since $m\geq 2$, 
$\alpha_1\not=\alpha_0$ and, by construction, there is a path from $\alpha_1$ to every element $\beta$ of the infinite set  $\Sigma_2$ (consisting of the vertices between 
$\alpha_1$ 
and $\beta$ in $q_{\beta}\in P_1$) such that  the only vertex in $p_1$ and this path is $\alpha_1$.  Set $P_2$ to the set of paths from $\alpha$ to $\beta\in \Sigma_2$ 
from the 
previous sentence. 

We may repeat the above process \emph{ad infinitum} to obtain for all $i>0$: $\beta_{i+1}\in \Sigma_{2i+1}$ and a path $p_{i+1}\in P_{2i+1}$ from $\alpha_i$ to $\beta_{i
+1}$, an $
\alpha_{i+1}$ in $p_{i+1}$, an infinite 
$\Sigma_{2i+2}\subseteq \Sigma_{2i+1}$ and an infinite set $P_{2i+2}$ of paths from $\alpha_{i+1}$ to every element of $\Sigma_{2i+2}$ such that the only vertex in 
$p_{i+1}$ and 
any path in $P_{2i+2}$ is $\alpha_{i+1}$. 

Hence there is a walk $\bf{r}$ containing $\set{\alpha_i}{i\in\N}$ consisting of the vertices on the paths $p_{i+1}$ between $\alpha_i$ and $\alpha_{i+1}$. In fact, by 
construction, the 
only vertex on both $p_{i}$ and $p_{i+1}$ is $\alpha_{i+1}$, and so the walk $\bf{r}$ is a ray. 
Moreover, there are infinitely many paths from $\bf{r}$ to $\Sigma$ consisting of the remaining vertices  on $p_{i+1}$ between $\alpha_{i+1}$ and $\beta_{i+1}$. Again by 
construction 
the only vertex on both $p_{i}$ and $p_{i+1}$ is $\alpha_{i+1}$ and so the paths from $\alpha_{i+1}\in \bf{r}$ to $\beta_{i+1}\in\Sigma$ are disjoint.
\end{proof}



\section{The ends of a semigroup}\label{ends of a semigroup}

Throughout this section, we let $S$ be a finitely generated semigroup and let $A$ be any finite generating set for $S$. 
The 
\emph{right Cayley graph $\Gamma_{r}(S,A)$} of $S$ with respect to $A$ is the directed
graph with vertex set $S$ and edges 
$(s, sa)\in \Gamma_{r}(S,A)$ for all $s\in S$ and for all $a\in A$. We refer to $a$ as the \emph{label} of the edge $(s, sa)$. 
The \emph{left Cayley graph}
$\Gamma_l(S,A)$ is defined dually. 

If $S$ is a semigroup, then the dual $S^*$ of $S$ is just the set $S$ with multiplication $*$ defined by $x* y=yx$ for all $x,y\in S$. 
It follows directly from the definition that $\Gamma_l(S,A)=\Gamma_r(S^{\ast},A)$. Therefore to understand the end structure of  a semigroup it suffices to study right 
Cayley graphs 
only. 

We require the following lemma to prove the results in this section.


\begin{lem} \label{path_replace}
Let $S$ be a semigroup, let $T$ be a subsemigroup of $S$ generated by a finite set $A$, and let $s\in S$. Suppose that $|\genset{T, s}\setminus T|=n\in\N$. 
Then there exists $N\in \N$ such that  for all $b_1, b_2, \ldots, b_{n}\in A\cup\{s\}$  if
$s, sb_1, \ldots, sb_1\cdots b_n$ are distinct, then there exists $i\leq n$ and $a_1, a_2, \ldots, a_j\in A$ such that 
$j\leq N$ and $sb_1\cdots b_i=a_1\cdots a_j$. 
\end{lem}
\proof 
Let $X:=\set{s c_1c_2\ldots c_i\in T}{c_j \in A \cup \{s\},\ 1\leq i \leq n}$. Then $X$ is finite and so 
there exists $N\in\N$ such that every element of $X$ can be given as a product of elements of $A$ of length at most $N$.  By the pigeonhole principle, there exists $i$ 
such that 
$sb_1\cdots b_i\in T$   and hence $sb_1\cdots b_i\in X$. It follows that  there exist $a_1, \ldots, a_j\in A$ such that $j\leq N$ and 
$sb_1\cdots b_{i}=a_1\cdots a_j$, as required.
\qed


\begin{prop}\label{master}
Let $S$ be a semigroup, let $T$ be a subsemigroup of $S$ generated by a finite set $A$, and let $s\in S$. If $\genset{T, s}\setminus T$ is finite, then 
$\Omega\Gamma_r(T, A)$ is isomorphic (as a partially ordered set) to $\Omega\Gamma_r(\genset{T, s}, A\cup\{s\})$.
\end{prop}
\proof For the sake of brevity, we denote $\Gamma_r(\genset{T, s}, A\cup\{s\})$ by $\Gamma$. We  use  
$\preccurlyeq$ to denote the preorder defined above on the rays and anti-rays of $\Gamma$.
We prove the proposition in two steps. The first step is to show that every ray or anti-ray in $\Gamma$ is 
equivalent to a ray or anti-ray with vertices in $T$ and edges labelled by elements of $A$. 
The second step is to show that if  $\r$ and $\r'$ are rays or anti-rays with vertices in $T$, edges labelled by 
elements of $A$, and $\r\preccurlyeq \r'$, then  there exist infinitely many disjoint paths from $\r'$ to $\r$ with 
edges labelled by elements of $A$. So, the first step ensures that every end $\omega$ of  $\Gamma$, contains a ray or anti-ray $\r_{\omega}$ with vertices in $T$ and 
edges labelled 
by elements of $A$. The 
second step implies that the mapping $\Psi:\Omega\Gamma \to  \Omega\Gamma_r(T, A)$ defined so that  
$\Psi(\omega)$ equals the end of $\Omega\Gamma_r(T, A)$ containing $\r_{\omega}$ is an isomorphism. 
We only give the proof of these steps for rays, an analogous argument can be used for anti-rays. 

Let $U=\genset{T, s}\setminus T$, let $n=|U|$, and 
let $\mathbf{r}=(x, xb_1, xb_1b_2, \ldots)$ be a ray in  $\Gamma$ for some 
$b_1, b_2, \ldots\in A\cup \{s\}$ and $x\in \genset{T, s}$. Since  
$U$ is finite, only finitely many elements of $\mathbf{r}$ can lie in $U$, and so we may assume 
without loss of generality that $x, xb_1, xb_1b_2, \ldots\in T$.  If $b_i\not=s$ for all $i$, then there is nothing to prove. 

If $b_k$ is the first occurrence of $s$ in $\{b_1, b_2, \ldots\}$, then since the vertices of $\r$ are distinct so are the elements $b_k, b_kb_{k+1}, \ldots, b_kb_{k+1}\cdots 
b_{k+n}$. 
Hence by Lemma \ref{path_replace} there exist $i, N\in \N$ and $a_1, a_2, \ldots, a_j\in A$ such that $j\leq N$ and $b_kb_{k+1}\cdots b_{k+i}=sb_{k+1}\cdots b_{k+i}
=a_1\cdots a_j$. 
Hence 
$$\mathbf{w}_{0}=(xb_1\cdots b_{k-1}, xb_1\cdots b_{k-1}a_1, \ldots, xb_1\cdots b_{k-1}a_1\cdots a_j)$$
is a walk in $\Gamma$ with vertices in $T$ and edges labelled by elements of $A$. 
We repeatedly apply Lemma \ref{path_replace} to successive occurrences of $s$ in $\{b_1, b_2, \ldots\}$
 to obtain finite walks $\mathbf{w}_1$, $\mathbf{w}_2,\ldots$  with vertices in $T$ and edges labelled by elements of $A$.  
 The length of $\mathbf{w}_{i}$ 
is bounded by $N$ for all $i\in \N$ and the final vertices are distinct, and hence by Lemma \ref{concatenation} every vertex in the sequence
$\mathbf{w}_{0}^\smallfrown\mathbf{w}_{1}^\smallfrown\cdots$ occurs only finitely many times. 
Let $\mathbf{w}$ be the walk obtained by replacing the subpaths of $\r$ by the $\mathbf{w}_i$. 
Every vertex of $\mathbf{w}$  not in some $\mathbf{w}_{i}$ occurs only once, since $\r$ is a ray.  Hence every 
vertex of $\mathbf{w}$ occurs only finitely many times, and 
so by Lemma \ref{walktoray} there  is a subray $\mathbf{r}'$ of $\mathbf{w}$ such that  $\mathbf{r}\approx \mathbf{r}'$, as required.

For the second step of the proof, let $\mathbf{r}$ and $\mathbf{r}'$ be  rays in $\Gamma$ with vertices in $T$, 
edges labelled by 
elements of $A$, and $\r\preccurlyeq \r'$. Since $\r\preccurlyeq \r'$, there exist infinitely many disjoint paths in 
$\Gamma$ from $\r'$ 
to $\r$. We may assume without loss of generality that there are at least $n$ vertices in each of these paths after 
the last occurrence of $s$ as an edge label. Hence, by repeatedly applying Lemma \ref{path_replace}, there 
exists $N\in \N$ and  
infinitely many paths from $\r'$ to $\r$ labelled by elements of $A$. Moreover,  there is a path of length at most 
$N$ from every element in one of the new paths to  some element in the original path it was obtained from by 
applying Lemma \ref{path_replace}. If infinitely many of these new paths are disjoint, then there is 
nothing to prove. Otherwise infinitely many of these paths have non-empty intersection with a finite subset of $T$, 
and so infinitely many paths contain some fixed element $t \in T$.  Hence there are path of length at most $N$ 
from $t$ to infinitely many vertices in the original paths, which contradicts the out-local finiteness of $\Gamma$. 
\qed\vspace{\baselineskip}


\begin{cor}\label{change_gen}
Let $S$ be a finitely generated semigroup and let $A$ and $B$ be any finite generating sets for $S$. Then 
$\Omega\Gamma_r(S,A)$ is isomorphic (as a partially ordered set) to $\Omega\Gamma_r(S,B)$.
\end{cor}
\proof 
It suffices to show that $\Omega\Gamma_r(S,A)$ is isomorphic to $\Omega\Gamma_r(S,A\cup{\{s\}})$ for any $s\in S$, since then 
$\Omega\Gamma_r(S,A)$ is isomorphic to  $\Omega\Gamma_r(S,A\cup B)$ is isomorphic to $\Omega\Gamma_r(S,B)$, as required.  Certainly $S$ is a finitely 
generated 
subsemigroup of $S$ such that $\genset{S, s}\setminus S$ is finite, and so it follows by Proposition \ref{master} that $\Omega\Gamma_r(S,A)$ is isomorphic to $\Omega
\Gamma_r(S,A
\cup{\{s\}})$, as required. 
 \qed\vspace{\baselineskip}


Following from Corollary \ref{change_gen} we define $\Omega S=\Omega\Gamma_{r}(S, A)$ for any finite generating set $A$ of 
$S$.  We  refer to $\Omega S$ as the \emph{ends of $S$}. 

Note that if $S$ is a finitely generated group, then it follows by Hopf's Theorem \cite[Satz II ]{Hopf} that the ends of $S$ form 
an anti-chain with $1$, $2$, or $2^{\aleph_0}$  elements. 

In section 5 we give examples of finitely generated semigroups  with any finite number  or $\aleph_0$ ends 
(Examples \ref{ReesCharact}, \ref{ex:GreenFail}). Any group with $2^{\aleph_0}$ group ends will also have $2^{\aleph_0}$ ends as a semigroup. It is easy to see that the free 
monoid on two generators will have $2^{\aleph_0}$ ends as all pairs of rays are incomparable. It is not 
known whether, in the absence of the Continuum Hypothesis,  there exists a finitely generated semigroup 
$S$ such that $\Omega S$ has $\kappa$ elements where $\aleph_0<\kappa< 2^{\aleph_0}$. The  question of which posets can occur as the partial order of ends $\Omega S$ of some finitely generated 
semigroup $S$ is unresolved. 

If $S$ is a semigroup and $T$ is a subsemigroup of $S$, then the \emph{Rees index} of $T$ in $S$ is just 
$|S\setminus T|+1$.  


\begin{cor}\label{th:Rees}
Let $S$ be a finitely generated semigroup and let $T$ be a subsemigroup of $S$ of finite Rees
index. Then the partial order $\Omega S$ of the ends of $S$  is isomorphic to the partial order $\Omega T$ of the ends of $T$. 
\end{cor}
\proof 
Since $S$ is finitely generated, it follows by \cite[Theorem 1.1]{NikRees}, that $T$ is finitely generated. Let $A$ be any finite generating set for $T$ and let 
$s\in S\setminus T$ be arbitrary. 
Then  $\genset{T, s}\setminus T\subseteq S\setminus T$ and so   $\genset{T, s}\setminus T$ is finite. It follows from Proposition 
\ref{master} that $\Omega\Gamma_r(T, A)$ is isomorphic $\Omega\Gamma_r(\genset{T, s}, A\cup\{s\})$, and hence $\Omega T$ is 
isomorphic to $\Omega\genset{T, s}$.  Since $T\lneq \genset{T, s}\leq S$, by repeating this process (at most $|S\setminus T|$ times) 
we have shown that $\Omega S$ is isomorphic to $\Omega T$. \qed


\section{The number of ends of a left cancellative semigroup} \label{left_cancel_sect}
In this section we prove that left cancellative semigroups can only have a restricted number of ends, unlike the 
general case (See Proposition \ref{ReesCharact}). 

 A semigroup $S$ is \emph{left cancellative} if  $x=y$ whenever $ax=ay$ where $a,x,y\in S$. \emph{Right 
 cancellative} is defined analogously.  A semigroup is \emph{cancellative} if it is both left and right cancellative. 

A left or right cancellative monoid contains only one idempotent (the identity). A left cancellative semigroup 
contains at most one idempotent in every $\lrel$-class, and the analogous statement holds for right cancellative 
semigroups.    The structure of a cancellative semigroup $S$ is straightforward to describe: either $S$ is $\rrel$-
trivial or $S$ is a monoid with group of units $G$,  every  $\rrel$-class is of the form $xG$, and every $\lrel$-class 
is of the form $Gx$ for some $x\in S$ (see for example \cite{ReesStructure}). We start this section by giving an analogous 
description of the structure of a left  cancellative semigroup. It is possible to deduce these results from \cite{LeftCancSemi} although they are not couched in this notation, and so we 
include a proof for the sake of completeness. 


A \emph{right group} is the direct product of a group $G$ and right zero semigroup $E$.

\begin{thm}\cite[Theorem 1.27]{C+P}\label{RightGroupLCRSimp}
A semigroup is a right group if and only if it is left cancellative and $\rrel$-simple.
\end{thm}

\begin{prop}\label{left_cancel}
Let $S$ be a left cancellative semigroup and let $U$ be the set of regular elements in $S$. Then:
\begin{enumerate}
\item[\rm (i)] $S\setminus U$ is an ideal (in the case when $S$ is a group it is empty);
\item[\rm (ii)] if $U$ is non-empty, then $U$ is a right group;
\item[\rm (iii)]  if $x\in S$ has non-trivial $\rrel$-class $R_x$, then $R_x=xU$;
\item[\rm (iv)] if $x\in S$ is arbitrary and $U$ is non-empty, then $xU$ is an $\rrel$-class of $S$ (not necessarily containing $x$).
\end{enumerate}
\end{prop}
\proof 
Let $x,y \in S$ and assume that $xy$ is a regular element. It follows that there exists $z\in S$ such that $xyzxy=xy$. By cancelling we see that $yzxy=y$ and hence $y
$ must be a regular element. From $yzxy=y$ we must have $yzxyzx=yzx$, again by cancelling we see $xyzx=x$ and hence $x$ is also a regular element. This means 
that $S\setminus U$ is an ideal of $S$.

For the second part we assume that $U$ is non-empty. If $S$ contains a regular element then it contains an idempotent. Let $e$ and $f$ be idempotents in $S$. Then 
$e^2f=ef$ and $f^2e=fe$ and so, by cancelling,  $ef=f$ and $fe=e$. Thus $e\rrel f$ and, since every regular element is 
 $\rrel$-related to an idempotent,  the regular elements of $S$ are contained in a single $\rrel$-class of $S$. If $x$ and $y$ are regular, then there exists $x'\in S$ such 
that $xx'x=x$ and, since $y\rrel x'$, there exists $z\in S$ such that $yz=x'$. Hence $xyzxy=xx'xy=xy$ and so $xy$ is regular. Hence, $U$ is a subsemigroup of $S$ and 
by part one $S\setminus U$ is an ideal and so $U$ is $\rrel$-simple. It follows from Theorem \ref{RightGroupLCRSimp} that $U$ is a right group

We now proof parts three and four together. Let $x\in S$ and assume $U$ is non-empty. Let $e\in U$ be an idempotent. Then as all elements in $U$ are $\rrel$ related 
$xU$ is contained within an $\rrel$-class, say $R$. Let $y$ be an element of $R$ distinct from $xe$. Then there exists $s,t \in S$ such that $xes=y$ and $yt=xe$. It 
follows that $xest=xe$ and hence $xestst=xest$. By cancelling $(st)^2=st$ is an idempotent and as $S\setminus U$ is an ideal $s,t\in U$. This means that $R=xU$. If 
$x$ lies in a non-trivial $\rrel$-class then there exists $y\rrel x$ such that $y\neq x$ and there exists $s,t \in S$ such that $xs=y$ and $yt=x$. Then as before we see that 
$st$ is an idempotent and $x=xst$ so $x\in xU$.
\qed

\begin{lem}\label{1_infty}
A left cancellative semigroup $S$ has either one $\rrel$-class or infinitely many $\rrel$-classes. 
\end{lem}
\proof
We show that either $S$ is regular or  $(x, x^2)\not\in \rrel$ for some $x\in S$. Suppose that $(x,x^2)\in\rrel$ for all $x\in S$. Then there exists $s\in S^1$ such that 
$x^2s=x$. Hence $x^2st=xt$ and so $(xs)t=t$ for all $t\in S$. Hence $xs$ is a left identity for $S$ and so $xs$ is an idempotent and $x\rrel xs$. Thus $S$ is regular and 
so by \cite[Exercise 1.11.4]{C+P}  has only one $\rrel$-class. 
If there exists $x\in S$ such that $(x, x^2)\not\in \rrel$, then $(x^i, x^j)\not\in\rrel$ for all $i,j\in \mathbb{N}$ such that $i\not=j$. Hence $S$ has infinitely many 
$\rrel$-classes.  \qed


\begin{cor}\label{infinf}
If $S$ has infinitely many $\rrel$-classes at least one of which is infinite, then it has infinitely  many infinite $\rrel$-classes.
\end{cor}
\proof  Since there is at least one infinite $\rrel$-class in $S$,  that $\rrel$-class is of the form $yU$ for some $y\in S$, and $|yU|=|U|$ by left cancellativity, it follows that 
$U$ is infinite. 
From the proof of Lemma \ref{1_infty}, there exists $x\in S$ such that $(x^i, x^j)\not\in\rrel$ for all $i,j\in \mathbb{N}$ such that $i\not=j$. By Proposition \ref{left_cancel}, 
$x^iU$ is an $
\rrel$-class of $S$ for all $i\in \N$ and $|x^iU|=|U|$ and, in particular, $x^iU$ is infinite for all $i\in \N$. It suffices to show that the sets $x^iU$ are disjoint. Suppose to the 
contrary that 
$x^iU\cap x^jU\not=\emptyset$ for some $i,j\in \N$ with $i<j$. Then, by left cancellativity, $x^{j-i}U\cap U\not=\emptyset$ and so $x^{j-i}\in U$ since $S\setminus U$ is 
an ideal. 
Therefore infinitely many powers of $x$, namely, $x^{j-i}, x^{2j-2i}, \ldots$, are $\rrel$-related, contradicting our assumption.
\qed\vspace{\baselineskip}



Right groups are a special case of Rees matrix semigroups where $|I|=1$ and the multiplication matrix $P$ consists of identity elements. Hence as a corollary to 
Proposition \ref{ReesCharact} below we have.

\begin{cor}\label{right_group}
Let $G$ be a finitely generated group and let $E$ be a finite right zero semigroup. Then $|\Omega(G\times E)|=|\Omega G|$.
\end{cor}

%
%


\begin{lem}\label{noname}
Let $S$ be a finitely generated left cancellative semigroup with no infinite $\rrel$-classes. If the Cayley graph of $S$ with respect to any finite generating set 
contains a ray $\r$ and there is an $s\in S$ such that there are paths from infinitely many points in $\r$ to $s$, then $\Omega S$ is infinite. 
\end{lem}
\proof 
Let $A$ be any finite generating set for $S$ and let $\r=(r_0, r_1, \ldots)$ be a ray in $\Gamma_r(S, A)$. 
We may write $r_0$ as a product $a_1\cdots a_n$ of generators in $A$. 

Assume, seeking a contradiction, that $S$ has finitely many ends. Since $S$ is left cancellative,  $\r_i=(s^i, s^ia_1, \ldots, s^ia_1\cdots a_n=s^ir_0, s^ir_1, \ldots)$ is a 
ray for all $i\in\N$. 
Thus, by assumption, there exist $i,j\in \N$ such that $i<j$ and $\r_i\approx \r_j$. Again using the left cancellativity of $S$, it follows that $\r_{j-i}\approx \r$ and so there 
is a path from 
$s^{j-i}r_k$ to $r_l$ for some $k,l\in \N$.  There is a path from $s$ to $s^{j-i}r_k$ and hence to $r_l$.
But in this case,  $r_l\rrel r_{l+1}\rrel\cdots$ and so $S$ has an infinite $\rrel$-class, which is a contradiction. 
\qed\vspace{\baselineskip}


The main results of this section are given below. 
\begin{thm}\label{thm_cancel1}
Let $S$ be an infinite finitely generated left cancellative semigroup. Then $|\Omega S|=1,2$ or $|\Omega S|\geq \aleph_0$.
\end{thm}
\proof 
If $S$ has only one $\rrel$-class, then by \cite[Theorem 1.27]{C+P} it follows that $S\cong G\times E$ where $G$ is a group and $E$ is a right zero semigroup. Since $S
$ is finitely 
generated, it follows that $G$ is finitely generated and $E$ is finite. Hence, by Proposition \ref{right_group}, $|\Omega S|=|\Omega G|$ and  by Hopf's Theorem 
\cite[Satz I]{Hopf}, $|\Omega G|=1,2$ or $2^{\aleph_0}$. 

Suppose that $S$ has more than one $\rrel$-class. Then Lemma \ref{1_infty} implies that $S$ has infinitely many $\rrel$-classes. If $S$ contains an infinite $\rrel$-class 
then then by Corollary \ref{infinf} $S$ contains infinitely many infinite $\rrel$-classes. By Konig's lemma each infinite $\rrel$-class contains a ray, none of these rays can be 
equivalent as the $\rrel$-classes are distinct. Thus $|\Omega S|\geq \aleph_0$.

Assume that $S$ has no infinite $\rrel$-class. Let $\Gamma$ denote the Cayley group of $S$ with respect to some finite generating set $A$ for $S$. 
 If $\Gamma$ contains a ray $\r$ and there is an $s\in S$ such that there are paths from infinitely many points in $\r$ to $s$, then $\Omega S$ is infinite by Lemma 
\ref{noname}.  If $\Gamma$ contains an anti-ray $\r$, then there exists $a\in A$ such that infinitely many of the elements in $\r$ are of the form $at$ for some $t\in S$. 
 In particular,  there is a path from $a$ to every $at$ in $\r$ and so by Lemma \ref{dominance} there exists a ray $\r'$ such that $\r\preccurlyeq \r'$. But then there are 
paths from 
infinitely many of the vertices in $\r'$ to any fixed element in $\r$, and so $\Omega S$ is infinite by Lemma \ref{noname}.

 Suppose that the Cayley graph of $S$ does not have the property of Lemma \ref{noname}. 
 Seeking a contradiction assume that $S$ has finitely many ends, and let $\r_1,\r_2, \ldots, \r_n$ be rays belonging in distinct ends such that the end containing $\r_1$ is 
minimal with 
respect to $\preccurlyeq$.  Since $\r_i\not\preccurlyeq \r_1$, there exists a finite  $F\subseteq S$ such that all paths from $\r_1$ to every $\r_i$ pass through $F$. By 
assumption there 
exists element $s$ in $\r_1$ such that there are no paths from $s$ to any element of $F$ and hence to any element in any $\r_i$. Since $S$ is left cancellative, $s\r_1$ 
and $s\r_2$ 
are rays. If $s\r_1\approx \r_i$ or $s\r_2\approx \r_i$, then  there is a path from $s$ to $\r_i$ and so $i=1$. In particular, $s\r_1\approx s\r_2$ and so, since $S$ is left 
cancellative, $
\r_1\approx \r_2$, a contradiction. We have shown that $S$ either has $1$ or infinitely many ends. 
\qed

\begin{cor}
Let $S$ be an infinite finitely generated  cancellative semigroup that is not a group. Then $|\Omega S|=1$ or $|\Omega S|\geq \aleph_0$.
\end{cor}
\proof Since $S$ is cancellative, it is certainly left cancellative and so $|\Omega S|=1, 2$, or $|\Omega S|\geq \aleph_0$ by Theorem \ref{thm_cancel1}.  
If $|\Omega S|=2$, then from the proof of Theorem \ref{thm_cancel1}, $S$ has only one $\rrel$-class and hence is a group. \qed


As mentioned above, it is not known what cardinalities $\Omega S$ can have, even for restricted types of 
semigroups such as those which are left cancellative. We prove that $\Omega S$ has cardinality $2^{\aleph_0}$ 
for a particular type of cancellative semigroup.
Ore's Theorem (see for instance \cite[Theorem 1.23]{C+P}) states that if a cancellative semigroup $S$ satisfies the condition that $sS\cap tS\not =\emptyset$ for all $s,t\in S$ then $S$ can be embedded in a group. 

\begin{thm}
A cancellative semigroup which cannot be embedded in a group has $2^{\aleph_0}$ ends.
\end{thm} 
\proof
Let $S$ be a cancellative semigroup that cannot be embedded in a group. As $S$ is not group-embeddable there exists $s,t\in S$ such that $sS\cap tS=\emptyset$.  Firstly we show that 
all elements of $\{s,t\}^*$ are distinct. Let $u=u_1u_2\ldots u_n,v=v_1,v_2\ldots v_m \in \{s,t\}^*$ and assume $u=_S v$, without loss of generality we assume the length of $u$ is less 
than or equal to the length of $v$. If $u$ is a prefix of $v$ then  $u=v=uv'$. It follows that $v'$ is a left identity for all elements of $S$. The first letter of $v'$ is (without loss of generality) 
$s$ and hence $tx=v'tx\in sS$ for all $x\in S$. If $u$ is not a prefix of $v$ then there exists a position $i\le n$ such that $u_j=v_j$ for all $j<i$ but $u_i\neq v_j$. As $u_j=v_j$ for all 
$j<i$ and $u=_S v$ it follows by left-cancellativity that $u_i\ldots u_n=v_i\ldots v_m$ and $u_i\neq v_j$, however, $sS\cap tS=\emptyset$ a contradiction. 

We now show that for $u,v\in \{s,t\}^*$ we have $v\in uS$ if and only if $u$ is a prefix of $v$. Clearly if $u$ is a prefix of $v$ then $v\in uS$. With the aim of getting a contradiction 
assume that $u=u_1u_2\ldots u_n$ is not a prefix of $v=v_1v_2\ldots v_m$ but $v\in uS$. This means there exists $x\in S$ such that $ux=v$. As $u$ is not a prefix of $v$ there exists 
$1\le i\le n$ such that $u_j=v_j$ for all $j<i$ but $u_i\neq v_j$. But then by left-cancellativity $u_i\ldots u_n x=_S v_i\ldots v_m$. Then as $\{u_i,v_i\}=\{s,t\}$ it follows that 
$sS\cap tS\neq \emptyset$. 

 Combining these facts gives a copy of the free semigroup on two generators as a subsemigroup of $S$ and there can be no paths between elements, this means $S$ has at least 
$2^{\aleph_0}$ ends. This is also the maximum possible number of ends so $|\Omega S|=2^{\aleph_0}$.
\qed


\section{Subsemigroups of finite Green index}


It follows from Proposition \ref{left_cancel} that if $T$ is a subsemigroup of a left cancellative semigroup $S$, then $\rrel^T=\rrel^V$ where $V$ is the right group of 
regular elements of 
$T$. 

Let $S$  be a semigroup and let $T$ be a subsemigroup of finite Green index.  It was shown in \cite{CGR} that $S$ is 
finitely generated if and only if $T$ is finitely generated. 
If $T$ is a submonoid of a left-cancellative monoid $S$ and $T$ has finite Green index in $S$, then, since the complement is an 
ideal, the group of units of $T$ has finite index in the 
group of units of $S$.


\begin{lem}\label{RightGroupGreenSubsemi}
If $S=G\times E$ is a right group, $G$ is infinite and $T$ is a subsemigroup of finite Green index then $T=H\times E$ 
is a right group where $H$ is of finite index in $G$.
\end{lem}
\begin{proof}
One can see $S$ has only one $\rrel^S$-class, therefore the $\mathcal{H}^S$-classes of $S$ are the 
$\mathcal{L}^S$-classes of $S$. As $(g,e)\cdot (h,f)=(gh,f)$ we see that $\mathcal{L}^S$-classes are of the form 
$G\times \{e\}$ for each $e\in E$.

If $T$ contains no elements of the form $(g,e)$ for some fixed $e\in E$ then the $\rrel^T$-class of each $(h,e)$ 
must be trivial. This follows as $(h,e)(g,f)$ can only be of the form $(hg,f)$ where $f\neq e$ and then there exists 
no element $(g',f')\in T$ such that $(hg,f)(g',f')=(h,e)$ as $T$ contains no elements of the form $(g,e)$. 

For each $e\in E$ we let $H_e$ be those elements $h\in G$ such that $(h,e)\in T$. We now show each $H_e$ 
contains $1_G$. Let $e\in E$. One can see $H_e$ is a subsemigroup of $G$ as in particular $(g,f)(h,e)=(gh,e)$ 
so $H_fH_e\subseteq H_e$. It is easy to see that a subsemigroup of finite Rees index in 
$G$ is equal to $G$ so we may assume $G\setminus H_e$ is infinite. As $\mathcal{H}^T$-classes are contained 
in $\mathcal{H}^S$-classes and as $G\setminus H_e$ is infinite we must have at least one non-trivial $\mathcal{H}^T$-class containing distinct elements $(g,e),(g',e)$ 
with $g,g'\not 
\in H_e$. As these elements are $\mathcal{H}^T$-related they are $\rrel^T$-related and hence there exists $(h,f),(h',f,)\in T$ such that $(g,e)(h,f)=(g',e)$ and $(g',e)
(h',f')=(g,e)$. This means $f=f'=e$ and furthermore that $ghh'=g$. It follows $hh'=1_G$ is an element of $H_e$. Hence, $H_e\subseteq H_f$ for all $e,f\in E$ so $H_e=H_f$ for all 
$e,f\in E$. We 
call this semigroup $H$.

As $H\times E$ has finite Green index in $G\times E$ it must follow that $H$ has finite Green index in $G$. It was shown in \cite[Corollary 34]{Bob+Nik} that if $H$ is a subsemigroup 
of finite index 
in a group $G$ then $H$ is a subgroup of $G$ with finite group index.
\end{proof}


\begin{lem}\label{simple_but_useful}
Let $S$ be a semigroup generated by  $A$ and let $T$ be a subsemigroup  of $S$ generated by $B$ with Green index $n\in\N$. If $s\in S$ and $a_1, a_2, \ldots, a_{m
+k}\in A$ such 
that the number of $\rrel^T$-classes containing any of $sa_1a_2\cdots a_{m+1}, sa_1a_2\cdots a_{m+2}, \ldots, sa_1a_2\cdots a_{m+k}$ is at least $n$, then there exist 
$i>m$ and 
$b_1, b_2, \ldots, b_j\in B$ such that $sa_1\cdots a_i=b_1\cdots b_j$. 
\end{lem} 

\proof 
If $sa_1a_2\cdots a_{m+1}, sa_1a_2\cdots a_{m+2}, \ldots, sa_1a_2\cdots a_{m+k}$ contains elements from $n$ $\rrel^T$-classes then $sa_1a_2\cdots a_{m+i}$ is an 
element of $T$ for some $i$. Any element of $T$ can be expressed over $B$ and hence there exists $b_1,b_2\ldots b_j\in B$ such that  $sa_1a_2\cdots a_{m+i}
=b_1b_2\cdots b_j$.

 \qed

\begin{thm}\label{thm_greens}
Let $S$ be a finitely generated left cancellative semigroup and let $T$ be a subsemigroup of $S$ of finite Green
index. Then $|\Omega S|=|\Omega T|$.
\end{thm}
\proof 
If $S$ is right simple, then $S\cong G\times E$ for some finitely generated group $G$ and $E$ is a finite right zero semigroup. Since 
$T$ has finite Green index in $S$, it follows that $T\cong H\times E$ where $H$ is a subgroup of finite index in $G$. In other words, $T$ is a right group and so $|\Omega 
T|=|\Omega H|
=|\Omega G|=|\Omega S|$ by Lemma \ref{right_group}.

Let $U$ be the right group of regular elements in $S$. 
  Since $S$ is finitely generated, it follows that $T$ is finitely generated. Let $A$ and $B$  be finite generating sets for $S$ and $T$, respectively, such that $B\subseteq 
A$.
Since $S\setminus U$ is an ideal, $U$ is also finitely 
 generated. Hence, as $T$ is also left cancellative, the right group of regular elements $V$ of $T$ is finitely generated.  It follows by Proposition \ref{left_cancel}
that $\rrel^T=\rrel^V$, and so $V$ has finite Green index in $U$.

Suppose that $S$ has more than one $\rrel$-class. Then, by Lemma \ref{1_infty},  $S$ has infinitely many $\rrel$-classes. If $S$ has 
no infinite $\rrel$-classes, then since $\rrel^T$-classes are contained in $\rrel^S$-classes, it follows that $T$ has finite Rees index in $S$ and so by Corollary 
\ref{th:Rees}, the 
theorem follows. 
We now consider the case that $S$ has infinitely many infinite $\rrel$-classes. By Proposition \ref{right_group}, $U$ either has $1$, $2$, or $2^{\aleph_0}$ ends. 

 If $U$ has $2^{\aleph_0}$ ends, then, since $S\setminus U$ is an ideal, $S$ has  $2^{\aleph_0}$ ends. 
  Since $V$ has finite Green index in $U$ and $U$ is a right group,  $V$  has $2^{\aleph_0}$ ends and so $T$ has $2^{\aleph_0}$ ends also. 
 
  Suppose that $U$ has $1$ or $2$ ends. Then $S$ and $T$ have at least $\aleph_0$ ends, since every pair of infinite $\rrel$-classes contain a pair of inequivalent rays. 
Let $
\Sigma(S)$ be the set of ends of $S$ containing a ray that has non-empty intersection with infinitely many $\rrel^S$-classes. By \cite[Lemma 2.8]{Zuther}, if $\omega$ is 
an end of $
\Gamma_r(S, A)$, then every ray in $\omega$ 
 is contained in a  strongly connected component or intersects infinitely many strongly connected components (but not both). 
 Since connected components of  $\Gamma_r(S, A)$ are precisely $\rrel^S$-classes, it follows that 
 $|\Omega S|=\max\{\aleph_0, |\Sigma(S)|\}$ and $|\Omega T|=\max\{\aleph_0, |\Sigma(T)|\}$. We conclude the proof by showing that $|\Sigma(S)|=|\Sigma(T)|$.
  
Let $\r$ be a ray or anti-ray in $\Gamma_r(S, A)$ that has non-empty intersection with 
 infinitely many $\rrel^S$-classes. Since every $\rrel^S$-class is a union of $\rrel^T$-classes, $\r$ has non-empty intersection with infinitely many $\rrel^T$-classes. 
Since there are 
only finitely many $\rrel^T$-classes in $S\setminus T$, we may assume without loss of generality that the elements in $\r$ are in $T$. 
  Let $n$ be the number of $\rrel^{T}$-classes in $S\setminus T$ and let $(xc_1, xc_1 c_2, \ldots, xc_1\cdots c_m)$ be a subpath of $\r$ 
  that has non-empty intersection with $n+1$, $\rrel^T$-classes. By left cancellativity, the path 
 $(c_1, c_1 c_2, \ldots, c_1\cdots c_m)$ has non-empty intersection with at least $n+1$ $\rrel^T$-classes also. 
 It follows that there exists $i$ such that $c_1\cdots c_i\in T$. Hence $c_1\cdots c_i$ is a product $b_1b_2\cdots b_j$ of elements in
  the generating set $B$ for $T$. Recursively replacing every such path  $(xc_1, xc_1 c_2, \ldots, xc_1\cdots c_i)$ by the corresponding walk $(xb_1, xb_1b_2, \ldots, 
xb_1\cdots b_j)$ 
we obtain a walk $\mathbf{w}=(w_0, w_1, \ldots)$ in $\Gamma_r(T, B)$ that has non-empty intersection with infinitely many $\rrel^T$-classes contained in $T$. 
If $i<j$ and $w_i\rrel^T w_j$, then $w_i\rrel^T w_{i+1}\rrel^T \cdots\rrel^T w_{j}$. But $\mathbf{w}$ has non-empty intersection with infinitely many $\rrel^T$-classes and 
so every 
vertex of $\mathbf{w}$ occurs only finitely many times. Hence, by Lemma \ref{walktoray}, $\mathbf{w}$ is equivalent to a ray or anti-ray in $\Gamma_r(T, B)$. 

Let $\r_1$ be a ray or anti-ray and let $\r_2$ be a ray or anti-ray in  $\Gamma_r(T, B)$ such that $\r_1$ and $\r_2$ have non-empty intersection with 
 infinitely many $\rrel^S$-classes. 
If $\r_1$ is equivalent to $\r_2$ in $\Gamma_r(T, B)$, then clearly $\r_1$ is equivalent to $\r_2$ in $\Gamma_r(S, A)$.  
Suppose that $\r_1$ is equivalent to $\r_2$ in $\Gamma_r(S, A)$. In this case, there are infinitely many disjoint paths from $\r_1$ to $\r_2$ and vice versa. By repeatedly 
applying 
Lemma \ref{simple_but_useful}, there exist infinitely many paths from $\r_1$ to $\r_2$ labelled by elements of $B$. If infinitely many of these paths are disjoint, then the proof is 
complete. 
Otherwise 
infinitely many of these paths have non-empty intersection with a finite subset of $S$, and so infinitely many paths contain some fixed element $s\in S$. But then there 
exists a path 
from $s$ to element in $\r_2$ and a path from that vertex  to an element in $\r_1$, and so infinitely many elements in $\r_1$ are $\rrel^S$-related, a contradiction. We 
have shown that 
for all rays or anti-rays $\r_1$ and $\r_2$ in $\Gamma_r(T, B)$ such that $\r_1$ and $\r_2$ have non-empty intersection with 
 infinitely many $\rrel^S$-classes, $\r_1$ is equivalent to $\r_2$ in $\Gamma_r(T,B)$ if and only if they are equivalent in $\Gamma_r(S,A)$. Therefore $|\Sigma(S)|=|
\Sigma(T)|$, as 
required. 
 \qed\vspace{\baselineskip}


\section{Examples}

In this section we give several examples of finitely generated semigroups $S$ and describe $\Omega S$ for these examples. 

The following example shows that unlike in the groups case it is possible for a left cancellative semigroup to have $\aleph_0$ ends.
\begin{ex}\label{NxN}
The semigroup $\N_0 \times \N_0$ under componentwise addition has $\aleph_{0}$ ends. For the sake of brevity we use $\Gamma$ to denote the Cayley graph $\Gamma_r(\N_0 \times \N_0,\{(0,1),(1,0)\})$. We show that any ray in  $\Gamma$ is equivalent to one of $$((i,0),(i,1),(i,2),\ 
\ldots ),\ ((0,i),(1,i),(2,i), \ldots)\text{ or }((0,0),(1,0),(1,1),(2,1),(2,2)\ldots)$$ for each $i\in \N_0$. We first note that there are no anti-rays in $\Gamma$. Any ray 
either contains finitely many elements in the first component of its vertices, finitely many elements in the second component of its vertices or infinitely many distinct 
elements in both components. In the first case as elements are eventually of the form $(i,j)$ for some fixed $i$ the ray is equivalent to $((i,0),(i,1),(i,2), \ldots )$. Equivalently if the ray has finitely many elements in the second component of its vertices then it will be equivalent to some  $((0,i),(1,i),(2,i), \ldots)$. In the case that the ray has infinitely many distinct elements in both components then for any element $(i,j)$ where $i<j$ there is a path from $(i,i)$ to $(i,j)$ to $(j,j)$ and we see that the ray is equivalent to $((0,0),(1,0),(1,1),(2,1), (2,2),
\ldots)
$.
\end{ex}

The following example demonstrates the existence of anti-rays which are not equivalent to any ray. It also shows that it is possible to have anti-rays in a semigroup with trivial $\rrel$-classes.

\begin{ex}\label{ex:kosarka}
Let $M$ be the monoid $\langle a,b| aba=b \rangle$. It is easy to check that $aba\to b$ and $b^2a\to ab^2$ is a complete rewriting system.  In a similar way to Example 
\ref{NxN} we can show that this monoid has $\aleph_0$ ends.

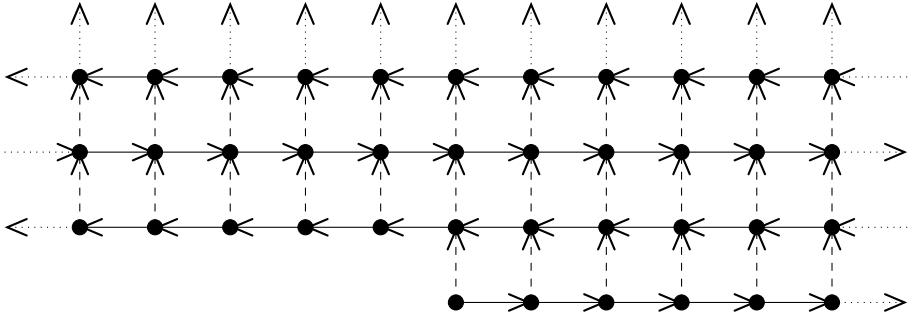
\begin{figure}
\caption{A portion of the right Cayley graph of  $\langle a,b\:|\: aba=b \rangle$ from Example \ref{ex:kosarka}, edges labelled by $a$ are represented with solid lines and those labelled by $b$ with dashed lines.}
\begin{center}
\begin{tikzpicture}[>=angle 45]
\draw[fill] (0,0) circle (0.1cm);
\draw[fill] (1,0) circle (0.1cm);
\draw[fill] (2,0) circle (0.1cm);
\draw[fill] (3,0) circle (0.1cm);
\draw[fill] (4,0) circle (0.1cm);
\draw[fill] (5,0) circle (0.1cm);

\draw[fill] (0,1) circle (0.1cm);
\draw[fill] (1,1) circle (0.1cm);
\draw[fill] (2,1) circle (0.1cm);
\draw[fill] (3,1) circle (0.1cm);
\draw[fill] (4,1) circle (0.1cm);
\draw[fill] (5,1) circle (0.1cm);
\draw[fill] (-1,1) circle (0.1cm);
\draw[fill] (-2,1) circle (0.1cm);
\draw[fill] (-3,1) circle (0.1cm);
\draw[fill] (-4,1) circle (0.1cm);
\draw[fill] (-5,1) circle (0.1cm);

\draw[fill] (0,2) circle (0.1cm);
\draw[fill] (1,2) circle (0.1cm);
\draw[fill] (2,2) circle (0.1cm);
\draw[fill] (3,2) circle (0.1cm);
\draw[fill] (4,2) circle (0.1cm);
\draw[fill] (5,2) circle (0.1cm);
\draw[fill] (-1,2) circle (0.1cm);
\draw[fill] (-2,2) circle (0.1cm);
\draw[fill] (-3,2) circle (0.1cm);
\draw[fill] (-4,2) circle (0.1cm);
\draw[fill] (-5,2) circle (0.1cm);

\draw[fill] (0,3) circle (0.1cm);
\draw[fill] (1,3) circle (0.1cm);
\draw[fill] (2,3) circle (0.1cm);
\draw[fill] (3,3) circle (0.1cm);
\draw[fill] (4,3) circle (0.1cm);
\draw[fill] (5,3) circle (0.1cm);
\draw[fill] (-1,3) circle (0.1cm);
\draw[fill] (-2,3) circle (0.1cm);
\draw[fill] (-3,3) circle (0.1cm);
\draw[fill] (-4,3) circle (0.1cm);
\draw[fill] (-5,3) circle (0.1cm);

\draw[ar] (0,0)--(1,0);
\draw[ar] (1,0)--(2,0);
\draw[ar] (2,0)--(3,0);
\draw[ar] (3,0)--(4,0);
\draw[ar] (4,0)--(5,0);
\draw[ar,dotted] (5,0)--(6,0);

\draw[ar] (0,1)--(-1,1);
\draw[ar] (-1,1)--(-2,1);
\draw[ar] (-2,1)--(-3,1);
\draw[ar] (-3,1)--(-4,1);
\draw[ar] (-4,1)--(-5,1);
\draw[ar,dotted] (-5,1)--(-6,1);
\draw[ar] (1,1)--(0,1);
\draw[ar] (2,1)--(1,1);
\draw[ar] (3,1)--(2,1);
\draw[ar] (4,1)--(3,1);
\draw[ar] (5,1)--(4,1);
\draw[ar,dotted] (6,1)--(5,1);

\draw[ar] (0,2)--(1,2);
\draw[ar] (1,2)--(2,2);
\draw[ar] (2,2)--(3,2);
\draw[ar] (3,2)--(4,2);
\draw[ar] (4,2)--(5,2);
\draw[ar,dotted] (5,2)--(6,2);
\draw[ar] (-1,2)--(0,2);
\draw[ar] (-2,2)--(-1,2);
\draw[ar] (-3,2)--(-2,2);
\draw[ar] (-4,2)--(-3,2);
\draw[ar] (-5,2)--(-4,2);
\draw[ar,dotted] (-6,2)--(-5,2);

\draw[ar] (0,3)--(-1,3);
\draw[ar] (-1,3)--(-2,3);
\draw[ar] (-2,3)--(-3,3);
\draw[ar] (-3,3)--(-4,3);
\draw[ar] (-4,3)--(-5,3);
\draw[ar,dotted] (-5,3)--(-6,3);
\draw[ar] (1,3)--(0,3);
\draw[ar] (2,3)--(1,3);
\draw[ar] (3,3)--(2,3);
\draw[ar] (4,3)--(3,3);
\draw[ar] (5,3)--(4,3);
\draw[ar,dotted] (6,3)--(5,3);

\draw[ar,dashed] (0,0)--(0,1);
\draw[ar,dashed] (0,1)--(0,2);
\draw[ar,dashed] (0,2)--(0,3);
\draw[ar,dashed,dotted] (0,3)--(0,4);

\draw[ar,dashed] (1,0)--(1,1);
\draw[ar,dashed] (1,1)--(1,2);
\draw[ar,dashed] (1,2)--(1,3);
\draw[ar,dashed,dotted] (1,3)--(1,4);

\draw[ar,dashed] (2,0)--(2,1);
\draw[ar,dashed] (2,1)--(2,2);
\draw[ar,dashed] (2,2)--(2,3);
\draw[ar,dashed,dotted] (2,3)--(2,4);

\draw[ar,dashed] (3,0)--(3,1);
\draw[ar,dashed] (3,1)--(3,2);
\draw[ar,dashed] (3,2)--(3,3);
\draw[ar,dashed,dotted] (3,3)--(3,4);

\draw[ar,dashed] (4,0)--(4,1);
\draw[ar,dashed] (4,1)--(4,2);
\draw[ar,dashed] (4,2)--(4,3);
\draw[ar,dashed,dotted] (4,3)--(4,4);

\draw[ar,dashed] (5,0)--(5,1);
\draw[ar,dashed] (5,1)--(5,2);
\draw[ar,dashed] (5,2)--(5,3);
\draw[ar,dashed,dotted] (5,3)--(5,4);

\draw[ar,dashed] (-1,1)--(-1,2);
\draw[ar,dashed] (-1,2)--(-1,3);
\draw[ar,dashed,dotted] (-1,3)--(-1,4);

\draw[ar,dashed] (-2,1)--(-2,2);
\draw[ar,dashed] (-2,2)--(-2,3);
\draw[ar,dashed,dotted] (-2,3)--(-2,4);

\draw[ar,dashed] (-3,1)--(-3,2);
\draw[ar,dashed] (-3,2)--(-3,3);
\draw[ar,dashed,dotted] (-3,3)--(-3,4);

\draw[ar,dashed] (-4,1)--(-4,2);
\draw[ar,dashed] (-4,2)--(-4,3);
\draw[ar,dashed,dotted] (-4,3)--(-4,4);

\draw[ar,dashed] (-5,1)--(-5,2);
\draw[ar,dashed] (-5,2)--(-5,3);
\draw[ar,dashed,dotted] (-5,3)--(-5,4);
\end{tikzpicture}
\end{center}

\end{figure}
\end{ex}

The following example demonstrates that in general a subsemigroup of finite Green index may have a different number of ends from the original semigroup.
\begin{ex}\label{ex:GreenFail}
Let $\{0,1\} $ be the semigroup with the usual multiplication (of real numbers).
Consider the semigroup $\mathbb{Z} \times \mathbb{Z} \times \{0,1\}$. Then $T=\mathbb{Z}\times \mathbb{Z}\times \{1\}$ is a subsemigroup and $\mathbb{Z}\times
\mathbb{Z}\times \{0\}$ is an $\mathcal{H}^T$-class in the complement. 
Hence $\mathbb{Z}\times \mathbb{Z}\times \{1\}$ has finite Green index in $\mathbb{Z} \times \mathbb{Z} \times \{0,1\}$. However, by inspection we see $\mathbb{Z} 
\times \mathbb{Z} \times \{0,1\}$ has $2$ ends corresponding to $\mathbb{Z} \times \mathbb{Z} \times \{1\}$ and $\mathbb{Z} \times \mathbb{Z} \times \{0\}$, however, $
\mathbb{Z} \times \mathbb{Z} \times \{1\}$ has only $1$ end.
For a diagram of a portion of the right Cayley graph of $\mathbb{Z} \times \mathbb{Z} \times \{0,1\}$ see Figure \ref{fig:GreenFail}. 

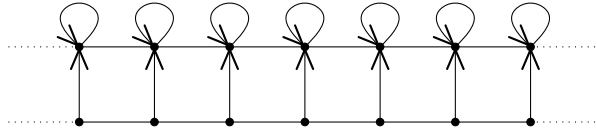
\begin{figure}
\caption{ \label{fig:GreenFail} A portion of the right Cayley graph of the semigroup defined by the presentation $\mathbb{Z}\times\{0,1\}$ from Example \ref{ex:GreenFail}.}
\begin{center}
\begin{tikzpicture}[>=angle 45]
\node at (0,0) [vertex] {};
\node at (0,1) [vertex] {};
\node at (1,0) [vertex] {};
\node at (1,1) [vertex] {};
\node at (2,0) [vertex] {};
\node at (2,1) [vertex] {};
\node at (3,0) [vertex] {};
\node at (3,1) [vertex] {};
\node at (-1,0) [vertex] {};
\node at (-1,1) [vertex] {};
\node at (-2,0) [vertex] {};
\node at (-2,1) [vertex] {};
\node at (-3,0) [vertex] {};
\node at (-3,1) [vertex] {};

\draw (-3,0)--(3,0);
\draw  (-3,1)--(3,1);
\draw [dotted] (3,0)--(4,0);
\draw [dotted] (-3,0)--(-4,0);
\draw [dotted] (3,1)--(4,1);
\draw [dotted] (-3,1)--(-4,1);

\draw[ar] (0,0)--(0,1); \myloop{(0,1)}{(0.2,1.5)}{(-0.2,1.5)}
\draw[ar] (1,0)--(1,1); \myloop{(1,1)}{(1.2,1.5)}{(0.8,1.5)}
\draw[ar] (2,0)--(2,1); \myloop{(2,1)}{(2.2,1.5)}{(1.8,1.5)}
\draw[ar] (3,0)--(3,1); \myloop{(3,1)}{(3.2,1.5)}{(2.8,1.5)}
\draw[ar] (-1,0)--(-1,1); \myloop{(-1,1)}{(-0.8,1.5)}{(-1.2,1.5)}
\draw[ar] (-2,0)--(-2,1); \myloop{(-2,1)}{(-1.8,1.5)}{(-2.2,1.5)}
\draw[ar] (-3,0)--(-3,1); \myloop{(-3,1)}{(-2.8,1.5)}{(-3.2,1.5)}

\end{tikzpicture}
\end{center}
\end{figure}
\end{ex}

Following Theorem \ref{thm_greens} one might question whether for a left cancellative semigroup it is possible to show that the end poset  of a subsemigroup of finite 
Green index is isomorphic to the end poset of the semigroup. The following example answers this in the negative.

\begin{ex}
Consider the semigroup $S=\mathbb{Z}\times\mathbb{Z}\times\mathbb{N}_0$ under componentwise addition. The subsemigroup $T=\mathbb{Z}\times\mathbb{Z}
\times(\mathbb{N}_0\setminus \{1\})$ is of finite Green index as the complement consists of $1$ $\mathcal{H}^T$-class. One can see that $S$ has $\aleph_0$ ends 
corresponding to each $\mathbb{Z}\times\mathbb{Z}\times\{i\}$ and to $\{0\}\times\{0\}\times\mathbb{N}_0$. In the poset of ends of $S$ any two elements are 
comparable. Either by inspection or by Theorem \ref{thm_greens} we see that $T$ also has $\aleph_0$ ends. However, the are no paths from $\mathbb{Z}\times
\mathbb{Z}\times \{2\}$ to $\mathbb{Z}\times\mathbb{Z}\times \{3\}$ or vice versa and hence the ends in these components cannot be comparable.
\end{ex}

The following proposition describes the left and right end posets of Rees matrix semigroups. As a corollary we see that for any $n,m\in\N$ there exists a semigroup with 
$n$ left ends and $m$ right ends.

Recall a Rees matrix semigroup $\mathcal{M}[G;I,\Lambda;P]$ has elements $I\times G \times \Lambda$ where $G$ is a group and $I$ and $\Lambda$ are index sets. Multiplication is defined by $(i,g,\lambda)(j,h,\mu)=(i,gp_{\lambda j}h,\mu)$ where $P=(p_{\lambda j})_{\lambda\in \Lambda,j\in I}$ is a $|\Lambda|\times |I|$ matrix over $G$. 

\begin{prop}\label{ReesCharact}
If $S$ is the Rees matrix semigroup $\mathcal{M}[G;I,\Lambda;P]$ where $I=\{i_1,i_2, \ldots , i_n\}$ and $\Lambda=\{\lambda_1,\lambda_2, \ldots \lambda_m\}$, $G$ is 
a finitely generated group and $P$ is a $m\times n$ matrix with entries in $G$ then the right ends of $S$ form an anti-chain of size $n\cdot |\Omega G|$ and the left ends of $S$ 
form an anti-chain of size $m\cdot |\Omega G|$.
\end{prop}
\begin{proof}
Let $X$ be a finite semigroup generating set for $G$ containing $1_G$ and let 

\begin{equation*}
A=\{(i,p_{\mu,j}^{-1}x,\lambda)| x\in X , \lambda,\mu \in \Lambda, i,j \in I\}.
\end{equation*}
\vspace{\baselineskip}
Clearly $A$ is a finite generating set for $S$.

Let $\Gamma_i$ be the induced subgraph of $\Gamma_r(S,A)$ on the vertices $\{i\} \times G \times \Lambda$ and let $\Gamma_{i,\lambda}$ be the subgraph of $
\Gamma_i$ with 
vertices $\{i\} \times G \times \{\lambda\}$ and edges with labels $(i,p_{\lambda,i}^{-1}x,\lambda)$. As $(i,g,\lambda)(j,h,\mu)=(i,gp_{\lambda,j}h,\mu)$ note that $
\Gamma_r(S,A)$ is the 
disjoint union of the $\Gamma_i$. This means that $\Omega\Gamma_r(S,A)$ is $n$ incomparable copies of $\Omega\Gamma_i$. As all ends in $\Omega G$ are 
incomparable it suffices to show that $\Omega \Gamma_i$ is isomorphic to $\Omega G$ for all $i \in I$.

We first note that for a fixed $\lambda \in \Lambda$, $\Gamma_{i,\lambda}$ is isomorphic to $\Gamma_r(G,X)$. We now prove that any ray in $\Gamma_i$ is 
equivalent to a
ray in $\Gamma_{i,\lambda}$, the proof for anti-rays is analogous. Let $\mathbf{r}=((i,g_0,\lambda_{j_0}),(i,g_1,\lambda_{j_1})\ldots )$ be a ray and let $\mathbf{r'}
=((i,g_0,\lambda),(i,g_1,\lambda)\ldots )$ be a sequence in $\Gamma_{i,\lambda}$. We show that there is an infinite walk $\mathbf{w}$ in $\Gamma_{i,\lambda}$ 
containing $\mathbf{r'}$ in which every vertex appears finitely often.

We construct $\mathbf{w}$ by concatenating the shortest paths in $\Gamma_{i,\lambda}$ between each $(i,g_k,\lambda)$ and $(i,g_{k+1},\lambda)$, these shortest 
paths exist 
because $\Gamma_{i,\lambda}$ is isomorphic to $\Gamma(G,X)$. Next we show that there is a global bound on the lengths of these shortest paths. If $(i,g,\mu)=(i,h,\nu)
(j,p_{\xi,k}^{-1}
x,\pi)$ then it follows $\mu=\pi$ and $g=hp_{\nu,j}p_{\xi,k}^{-1}x$. This means the shortest path in $\Gamma_{i,\lambda}$ between any consecutive elements of $
\mathbf{r'}$ is of 
length less than $K=\max\{|p_{j,\mu}p_{k,\nu}^{-1}|_X:j,k\in I, \mu,\nu \in \Lambda\}+1$. As $\mathbf{r}$ is a ray it follows there are at most $|\Lambda|$ repetitions of 
vertices in $
\mathbf{r'}$. Every vertex of $\mathbf{w}$ has a path of length less than $K$ to a vertex of $\mathbf{r'}$ and as $\Gamma_{i,\lambda}$ is out-locally finite this means that 
if some vertex $v$ appears infinitely often in $\mathbf{w}$ then infinitely many elements of $\mathbf{r'}$ can be reached from $v$ by a path of length less than or equal to 
$K$. But each vertex in $\mathbf{r'}$ appears at most $|\Lambda|$ times 
so any infinite set of elements of $\mathbf{r'}$ contains infinitely many vertices, a contradiction. By Lemma \ref{walktoray}, $\mathbf{w}$ contains a ray $\mathbf{s}$ with 
infinitely many 
disjoint paths from $\mathbf{s}$ to and from $\mathbf{r'}$ and hence to and from $\mathbf{r}$. 

This means any ray in $\Gamma_i$ is equivalent to a ray in $\Gamma_{i,\lambda}$, to complete the proof we must now verify that if we have rays $\mathbf{r_1}$ and $
\mathbf{r_2}$ in 
$\Gamma_{i,\lambda}$ such that $\mathbf{r}_1 \not\preccurlyeq \r_2$ then $\r_1 \not\preccurlyeq \r_2$ in $\Gamma_i$. Let $\mathbf{r_1}$ and $\mathbf{r_2}$ be incomparable rays 
in $\Gamma_{i,\lambda}$,  
as the rays are incomparable in $\Gamma_{i,\lambda}$ there exists a finite set $F=\{(i,f_1,\lambda),\ldots ,(i,f_m,\lambda)\}$ such that all paths from $\mathbf{r_1}$ to $
\mathbf{r_2}$ in 
$\Gamma_{i,\lambda}$ pass through $F$. For any edge $((i,g,\mu),(i,gp_{\mu,j}p_{\nu,k}^{-1}x,\xi))$ we have a word $w=x_1x_2\ldots x_p$ 
over $X$ 
of minimal length such that $w=_G  p_{\mu,j}p_{\nu,k}^{-1}x$ and a corresponding path

\begin{eqnarray*}
((i,g,\mu),(i,g,\lambda) ,(i,gx_1,\lambda),\ldots , (i,gx_1x_2\cdots x_p,\lambda),(i,gp_{\mu,j}p_{\nu,k}^{-1}x,\xi)).
\end{eqnarray*}

This means that any path in $\Gamma_i$ has a corresponding walk in $\Gamma_{i,\lambda}$ such that any point on the walk has a path of length less than $K+2$ to a 
vertex on the 
path in $\Gamma_i$. This means any path $\pi$ from $\mathbf{r_1}$ to $\mathbf{r_2}$ in $\Gamma_i$ has a corresponding walk in $\Gamma_{i,\lambda}$ and this must 
pass through 
$F$ and hence $\pi$ must be contain an element that can be reached from $F$ by a path of length less than or equal to $K+2$. As $\Gamma_i$ is out-locally finite there 
are only finitely many such elements so $\mathbf{r_1} \not\preccurlyeq \mathbf{r_2}$. 
\end{proof}


\begin{thebibliography}{3}

\bibitem{CGR}
  A. J.~Cain, R.~Gray and N.~R\v uskuc, Green index in semigroups: generators, presentations and automatic  structures, \emph{Semigroup Forum DOI 10.1007/s00233-012-9406-2 } ,(2009).
Semigroup Forum DOI 10.1007/s00233-012-9406-2 

\bibitem{C+P}
A. H.~Clifford and G. B.~Preston, The Algebraic Theory of Semigroups vol. 1, \emph{Mathematical Surveys, American Mathematical
Society, Providence, R.I.} (1961).

\bibitem{Dunwoody}
M.~J.~Dunwoody, The accessibility of finitely presented groups, \emph{Inventiones Mathematicae}, {\bf 81} (1985), no. 3, 449--457.

\bibitem{Bob+Nik}
R.~Gray and N.~Ru\v skuc, Green index and finiteness conditions for semigroups, \emph{J. Algebra} {\bf 320} (2008), 3145--3164.

\bibitem{GeoAutoSemi}
M.~Hoffmann and R.~M.~Thomas, A geometric characterization of automatic semigroups, \emph{Theoretical Computer Science} {\bf369} (2006), 300--313

\bibitem{Hopf}
H. ~Hopf, Enden offener {R}\"aume und unendliche diskontinuierliche {G}ruppen,\emph{Comment. Math. Helv.} {\bf 16}(1944),
81--100.

\bibitem{Vesna}
D.~A.~Jackson and V.~Kilibarda, Ends for monoids and semigroups, \emph{J. Aust. Math. Soc.} {\bf87} (2009), 101--127.

\bibitem{TropicalJ}
M.~Johnson and M.~Kambites, Green’s {$\mathcal{J}$}-order and the rank of tropical matrices, \emph{Journal of Pure and Applied Algebra} {\bf 217} (2013), 280--292.

\bibitem{Macpherson}
H.~Macpherson, Infinite distance transitive graphs of finite valency, \emph{Combinatorica} {\bf2} (1982), 63--69.

\bibitem{Muller+Schupp}
D.~E.~Muller and P.~E.~Schupp, Context-free languages, groups, the theory of ends, second-order logic, tiling problems, cellular automaton and vector addition systems, 
\emph{Bull. AMS} {\bf 4} (1981), 331--334.

\bibitem{ReesStructure}
D.~Rees, On the ideal structure of a semi-group satisfying a cancellation law, \emph{Quart. J. Math., Oxford Ser.} {\bf 19} (1948), 101--108.

\bibitem{Remmers}
J.~H.~Remmers, On the Geometry of Semigroups Presentations \emph{Advances in Mathematics} {\bf 36} (1980),
 283--296.

\bibitem{PursuitEvasion}
N. ~Robertson, P. ~Seymour and R. ~Thomas, Excluding infinite minors, \emph{Discrete Mathematics} {\bf95} (1991), 303--319.

\bibitem{NikRees}
N.~Ru\v skuc, On large subsemigroups and finiteness conditions of semigroups, \emph{Proc. London Math. Soc.} {\bf76} (1998), 383--405.

\bibitem{LeftCancSemi}
 M.~Satyanarayana, On left cancellative semigroups, \emph{Semigroup Forum} {\bf6} (1973), 317--329

\bibitem{Zuther}
J.~Zuther, Ends in digraphs, \emph{Discrete Math.} {\bf 184} (1998), 225--244.

\end{thebibliography}
\end{document}